\documentclass[12pt]{amsart}
\usepackage{amsmath, amssymb, amsfonts, amsthm, latexsym,enumerate}
\usepackage{graphicx}
\usepackage{tikz}

\newtheorem{Theorem}{Theorem}[section]
\newtheorem{Lemma}[Theorem]{Lemma}
\newtheorem{Corollary}[Theorem]{Corollary}
\newtheorem{Proposition}[Theorem]{Proposition}

\theoremstyle{definition}
\newtheorem{Remark}[Theorem]{Remark}
\newtheorem{Example}[Theorem]{Example}

\newtheorem{Conjecture}[Theorem]{Conjecture}

\def\Ass{\operatorname{Ass}}
\def\diam{\operatorname{diam}}
\def\depth{\operatorname{depth}}

\def\NN{{\mathbb N}}
\def\ZZ{{\mathbb Z}} 
\def\a{{\mathbf a}}
\def\b{{\mathbf b}}

\def\e{{\mathbf e}}
\def\0{{\mathbf 0}}
\def\1{{\mathbf 1}}
\def\mm{{\mathfrak m}}

\def\pp{{\mathfrak p}}

\def\E{{\mathcal E}}
\def\F{{\mathcal F}}

\def\D{{\Delta}}
\def\G{{\Gamma}}

\begin{document}

\title{Decreasing behavior of the depth functions of edge ideals}

\author{Ha Thi Thu Hien}
\address{Foreign Trade University, 91 Chua Lang, Hanoi, Vietnam}
\email{thuhienha504@gmail.com}

\author{Ha Minh Lam}
\address{Institute of Mathematics, Vietnam Academy of Science and Technology, 18 Hoang Quoc Viet, Hanoi, Vietnam}
\email{hmlam@math.ac.vn}

\author{Ngo Viet Trung}
\address{Institute of Mathematics, Vietnam Academy of Science and Technology, 18 Hoang Quoc Viet, Hanoi, Vietnam}
\email{nvtrung@math.ac.vn}


\begin{abstract}
Let $I$ be the edge ideal of a connected non-bipartite graph and $R$ the base polynomial ring. Then $\depth R/I \ge 1$ and $\depth R/I^t = 0$ for $t \gg 1$. We give combinatorial conditions for $\depth R/I^t = 1$ for some $t$ in between and show that the depth function is non-increasing thereafter. Especially, the depth function quickly decreases to 0 after reaching 1. We show that if $\depth R/I = 1$ then $\depth R/I^2 = 0$ and if $\depth R/I^2 = 1$ then $\depth R/I^5 = 0$. Other similar results suggest that if $\depth R/I^t = 1$ then $\depth R/I^{t+3} = 0$. This a surprising phenomenon because the depth of a power can determine a smaller depth of another power. Furthermore, we are able to give a simple combinatorial criterion for $\depth R/I^{(t)} = 1$ for $t \gg 1$ and show that the condition $\depth R/I^{(t)} = 1$ is persistent, where $I^{(t)}$ denotes the $t$-th symbolic powers of $I$.
\end{abstract}

\maketitle


\section*{Introduction}

Let $\G$ be a simple graph without isolated vertices on the vertex set $V = \{1,...,n\}$. Let $R = k[x_1,...,x_n]$ be the polynomial ring in $n$ variables over a field $k$. The {\em edge ideal} of $\G$ is the ideal
$$I := \big(x_ix_j|\ \{i,j\} \in \G\big).$$

By the work of Chen, Morey and Sung \cite{CMS}, $\depth R/I^t = 0$ for $t \gg 1$ if and only if $\G$ is a non-bipartite graph \cite{CMS}.  Later, Morales-Bernal, Morey and Villarreal \cite{MMV} showed that if $\depth R/I^t = 0$ then $\depth R/I^{t+1} = 0$. 
Since $\depth R/I \ge 1$, one may ask whether the depth function $\depth R/I^t $ takes the value 1 before reaching 0. 
This question has a negative answer. 
For instance, if $\G$ is a pentagon, $\depth R/I = \depth R/I^2 = 2$ and $\depth R/I^t = 0$ for $t \ge 3$ (see Example 2.8). 

If $\G$ is a connected bipartite graph, we know by the work of Simis, Vasconcelos, and Villarreal \cite{SVV} that 
$I^{(t)} = I^t$ for all $t \ge 1$, where $I^{(t)}$ denotes the $t$-th symbolic powers of $I$. 
Since symbolic powers always have positive depth, $\depth R/I^t \ge 1$ for all $t \ge 1$. Recently, T.N. Trung \cite{TNT} showed that $\depth R/I = 1$ for $t \gg 1$ and if $\depth R/I^t = 1$ then $\depth R/I^{t+1} = 1$.  

The above results are considered major advancements on the depth functions of edge ideals.
Inspired by these results we study the following problems for a connected non-bipartite graph:
\begin{itemize}
\item When does there exists a power $I^t$ with $\depth R/I^t = 1$ or $\depth R/I^{(t)} = 1$?
\item How does the depth function $\depth R/I^t $ or $\depth R/I^{(t)}$ behave after reaching 1?
\end{itemize}

It is well known that $\depth R/I^t = 1$ if and only if $H_\mm^0(R/I^t) = 0$ and $H_\mm^1(R/I^t) \neq 0$, where $H_\mm^i(R/I^t)$ denotes the $i$-th local cohomology modules with respect to the maximal ideal $\mm$ of $R$.
Since there is already a combinatorial criterion for $H_\mm^0(R/I^t) = 0$ in terms of $\G$ \cite{LT}, we will concentrate on the condition $H_\mm^1(R/I^t) \neq 0$. Actually, we study the above problems by replacing the condition $\depth R/I^t = 1$ by $H_\mm^1(R/I^t) \neq 0$.
\smallskip

Concerning the existence of a power $I^t$ with $H_\mm^1(R/I^t) \neq 0$ we obtain the following combinatorial conditions. Let $N[v]$ denote the closed neighborhood of a vertex $v$.
\medskip

\noindent {\bf Theorem \ref{non-dominating1}}.
Let $\G$ be a connected non-bipartite graph which has at least a non-dominating odd cycle.
Assume that there exist an odd cycle $C$ of length $2r+1$ and a vertex $v$ not adjacent to $C$ such that 
the induced subgraph $\G_{V \setminus  N[v]}$ is connected. Then  \par
{\rm (1)} $H_\mm^1(R/I^t) \neq 0$ for $t \ge r+1$ if $C = V \setminus N[v]$, \par
{\rm (2)} $H_\mm^1(R/I^t) \neq 0$ for $t \ge n-|N[v]|-r-1$ if $C \neq V \setminus N[v]$.
\medskip

\noindent {\bf Theorem \ref{dominating}}.
Let $\G$ be a connected non-bipartite graph which has only dominating odd cycles. Then $H_\mm^1(R/I^t) \neq 0$ for some $t \ge 1$ if and only if $\G$ has two disjoint maximal independent sets $F,G$ such that the induced graph $\G_{F \cup G}$ is connected. In this case, $H_\mm^1(R/I^t) \neq 0$ for $t \ge  |F|+|G|$.
\medskip

The above combinatorial conditions for $H_\mm^1(R/I^t) \neq 0$ are rather common. 
For instance, they are satisfied if $\G$ is an odd cycle or the union of an odd cycle with a path meeting the cycle at one end.
 \par

Following the terminology of \cite{HQ}, we say that the condition $H_\mm^1(R/I^t) \neq 0$ is persistent if $H_\mm^1(R/I^t) \neq 0$ implies $H_\mm^1(R/I^{t+1}) \neq 0$. If we have the persistence of $H_\mm^1(R/I^t) \neq 0$, then $\depth R/I^t = 1$ implies $\depth R/I^{t+1} \le 1$, which means that the depth function is non-increasing after reaching the value 1. This would be an evidence for the conjecture of Herzog and Hibi that the depth of powers of a squarefree monomials is a non-increasing function \cite{HH}. Though there are counter-examples \cite{HS,KSS}, this conjecture is still open for edge ideals of graphs.  

According to \cite{TT}, the condition $H_\mm^1(R/I^t) \neq 0$ can be divided in two technical cases (see Proposition \ref{TT} for details).
We will show the persistence of $H_\mm^1(R/I^t) \neq 0$ in one case (Proposition \ref{persistence1}), and we will give a mild condition for this persistence in the other case (Proposition \ref{persistence2}).  
These results solve the problem on the persistence of $H_\mm^1(R/I^t) \neq 0$ in a satisfactory manner.

Concerning the existence of a symbolic power $I^{(t)}$ with $\depth R/I^{(t)} = 1$ we are able to give the following criterion in terms of $\G$ and show that the condition $\depth R/I^{(t)} = 1$ is persistent, thereby solving the problems on symbolic powers completely. \medskip

\noindent{\bf Theorems \ref{symbolic1} and \ref{symbolic2}}. 
Let $\G$ be a connected graph. Then $\depth R/I^{(t)} = 1$ for some $t \ge 1$ if and only if $\G$ has two disjoint maximal independent sets $F,G$ such that the induced graph $\G_{F \cup G}$ is connected. Moreover, 
if $\depth R/I^{(t)} = 1$ then $\depth R/I^{(t+1)} = 1$.
 \medskip 
 
The condition of the above result is obviously satisfied if $\G$ is a bipartite graph. Hence, it is a natural 
extension of the results of T.N. Trung \cite{TNT} to arbitrary graphs.

Finally, we study the problem whether there exists a number $\delta$ independent of $\G$ such that if $H_\mm^1(R/I^t) \neq 0$ then $H_\mm^0(R/I^{t+\delta}) \neq 0$. Consequently, if $\depth R/I^t = 1$ then $\depth R/I^{t+\delta} = 0$.
The following results suggest that $\delta = 3$, which is a surprising phenomenon of the depth functions of edge ideals. As far as we know, this is the first case where the depth of a power determines a smaller depth of another power for a large class of ideals. A similar phenomenon occurs for Stanley-Reisner ideals, when $\depth R/I^t = \dim R/I^t$ for some $t \ge 3$ implies $\depth R/I^t = \dim R/I^t$ for all $t \ge 1$ \cite{MT,TT}.

\medskip

\noindent{\bf Theorem \ref{decrease1}}.
Let $\G$ be a non-bipartite graph.
If $\depth R/I = 1$, then $\depth R/I^2= 0$.
\medskip

\noindent{\bf Theorem \ref{square1}}.
Let $\G$ be a connected non-bipartite graph.
If $\depth R/I^2 = 1$, then $\depth R/I^5 = 0$.
\medskip

\noindent{\bf Theorem \ref{decrease}}.
Let $\G$ be a connected non-bipartite graph. Assume that there exists a vertex $v$ such that $\depth R_v/I_v^t = 0$, where $R_v := k[x_i|\ i \neq v]$ and $I_v := IR[x_v^{-1}]\cap R_v$.  Then $\depth R/I^{t+3} = 0$.
\medskip

Note that the assumption of Theorem \ref{decrease} is a sufficient condition for $H_\mm^1(R/I^t) \neq 0$. It covers the case dealt with in Theorem \ref{non-dominating1}. 


The paper is divided into 4 sections. The first section prepares basic results for the study of the aforementioned problems. The remaining sections deal with each of these problems separately.

We always assume that $\G$ is a simple graph without isolated vertices on the vertex set $V = \{1,...,n\}$.
Unless otherwise specified, we denote by $I$ the edge ideal of $\G$.
For basic concepts in graph theory we refer the reader to \cite{We}. 


\section{Preliminaries}

First we recall the combinatorial criterion for $\depth R/I^t = 0$ in \cite{LT}. Note that $\depth R/I = 0$ means $H_\mm^0(R/I) \neq 0$ or, equivalently, $\mm$ is an associated prime ideal of $I$.

Inspired of the well known notion of ear decomposition, 
we call a sequence of walks in $\G$ without repetition of the vertices except the endpoints a {\it generalized
ear decomposition} of $\G$ if the first walk is closed, the endpoints of each
subsequent walk are the only vertices of that walk belonging to earlier walks, and
the walks pass through all vertices of $\G$. 
Note that a walk without repetition of the vertices except the endpoints is a path or
a cycle or a repetitive edge (a closed walk of length 2). 
Generalized ear decompositions always exist in a connected
graph. In particular, every cycle can be used as the first walk of a generalized ear
decomposition. If the first walk of a generalized ear decomposition is an odd cycle, we call it an odd-beginning generalized ear decomposition. We refer to \cite{LT} for more details.
\par

If every connected component of $\G$ is non-bipartite, we call $\G$ a {\em strongly non-bipartite} graph.
For such a graph $\G$ we denote by $\varphi^*(\G)$ the minimal number of even walks in a family of odd-beginning generalized ear decompositions of each connected component of $\G$ and set
$$\mu^*(\G) : = (\varphi^*(\G)+n-c)/2,$$
where $c$ is the number of the connected components of $\G$.   
For instance, $\mu^*(\G) = t$ if $\G$ is an odd cycle of length $2t+1$.
The notation $\mu^*(\G)$ is inspired by the invariant $\mu(\G)$, which was introduced in coding theory by Sole and  Zaslavsky \cite{SZ} and studied in hypergraph theory by Frank \cite{Fr}.

A set $U \subseteq V$ is called {\em dominating} in $\G$ if every vertex of $V \setminus U$ is adjacent to a vertex of $U$. 
Let $\G_U$ denote the induced subgraph of $\G$ on $U$. 

\begin{Theorem} \label{LT} \cite[Theorem 3.6]{LT}
Let $\G$ be a strongly non-bipartite graph. 
Then $\depth R/I^t = 0$ if and only if there exists a dominating set $U$ of $\G$ such that 
$\G_U$ is strongly non-bipartite with $\mu^*(\G_U) < t$.
\end{Theorem}

From this criterion we immediately obtain the following result.

\begin{Corollary} \label{zero} \cite[Corollary 3.4]{CMS} \cite[Theorem 2.13]{MMV}
Let $I$ be the edge ideal of a strongly bipartite graph $\G$. Then $\depth R/I^t = 0$ for some $t \gg 1$.
If $\depth R/I^t = 0$, then $\depth R/I^{t+1} = 0$.
\end{Corollary}

Let $s(\G)$ denote the least number of $\mu^*(\G_U)$ among all dominating sets $U$ of $\G$ such that $\G_U$ is strongly non-bipartite. By Theorem \ref{LT}, $s(\G)+1$ is the least number $t$ such that $\depth R/I^t = 0$.

\begin{Example} \label{union} 
Let $\G$ be the union of a triangle $C$ and a path $P$ of length $s$ meeting $C$ at only one end; see Figure 1. 
Then $s+1$ is the least number such that $\depth R/I^s = 0$.
In fact, $\G$ has only an odd beginning generalized ear decomposition that consists of $C$ and $s$ edges of $P$ as closed walks of length $2$. From this it follows that $\phi^*(\G) = s$. Since $n = s+3$,
$$\mu^*(\G) = (s + s+3 -1)/2 = s+1.$$
Let $v$ be the end of $P$ not contained in $C$. A dominating sets $U$ of $\G$ with the above property is either $V \setminus \{v\}$ or $V$. Since $\G_{V \setminus \{v\}}$ is the union of $C$ with a path of length $s-1$, $\mu^*(\G_{V \setminus \{v\}}) = s$. Therefore, $s(\G) = s$.
\end{Example}

\begin{figure}[ht!]

\begin{tikzpicture}[scale=0.5] 

\draw [thick] (0,0) coordinate (a) -- (0,2) coordinate (b) ;
\draw [thick] (0,2) coordinate (b) -- (1.5,1) coordinate (c) ;
\draw [thick] (1.5,1) coordinate (c) -- (0,0) coordinate (a) ; 
\draw [thick] (1.5,1) coordinate (c) -- (3.5,1) coordinate (d);
\draw [thick] (3.5,1) coordinate (d) -- (5.5,1) coordinate (e);
\draw [thick] (5.5,1) coordinate (e) -- (7.5,1) coordinate (f);
\draw [thick] (7.5,1) coordinate (f) -- (9.5,1) coordinate (g);

\fill (a) circle (3pt);
\fill (b) circle (3pt);
\fill (c) circle (3pt);
\fill (d) circle (3pt);
\fill (e) circle (3pt);
\fill (f) circle (3pt);
\fill (g) circle (3pt);

\draw (5.5,2) node{$s=4$};
 
\end{tikzpicture}
\caption{}
\end{figure}

To estimate the $\mu^*$-invariant of strongly non-bipartite induced graphs we may use the following technique.
For a vertex $v \in V$, we set $U+v = U \cup \{v\}$.

\begin{Lemma} \label{extension}
Let $U$ be a subset of $V$ such that $\G_U$ is a strongly non-bipartite graph and $v$ an adjacent vertex of $U$. Then $\G_{U+v}$ is a strongly non-bipartite graph with 
$\mu^*(\G_{U+v}) \le \mu^*(\G_U)+1.$
\end{Lemma}

\begin{proof} 
Let $s = \varphi^*(\G_U)$. Then there exists an odd beginning generalized ear decomposition $\E$ of $\G_U$ such that $s$ is the number of even walks in $\E$.  
Let $u \in U$ be an adjacent vertex of $v$. Consider $\{u,v\}$ as a closed walk of length 2 and 
any other edge which connect a vertex of $U$ with $v$ as a walk of length 1.
If we add $\{u,v\}$ first and then the other edges in any order to $\E$, we obtain an odd beginning generalized ear decomposition $\F$ of $\G_{U+v}$. It is obvious that $s+1$ is the number of even walks in $\F$ and $|U+v| = |U|+1$. Using these facts we deduce from the definition of the $\mu^*$-invariant  that $\mu^*(\G_{U+v}) \le \mu^*(\G_U)+1.$
\end{proof}

For convenience, we say that a subgraph of $\G$ is dominating if the vertex set of the subgraph is dominating in $\G$. Moreover, we sometimes identify a subgraph with their vertex set.

\begin{Lemma} \label{triangle} 
Assume that $\G$ has a triangle $C$ and a path $P$ of length $s$ meeting $C$ at only one end such that the graph $C \cup P$ is dominating in $\G$. Then $\depth R/I^t = 0$ for $t \ge s+2$.
\end{Lemma}

\begin{proof} 
Let $v_0,v_1,...,v_s$ be the vertices of $P$ such that $\{x_i,x_{i+1}\} \in P$, $v_0 \in C$. By Proposition \ref{extension},
$$\mu^*(\G_{C + x_1+\cdots+x_{i+1}}) \le \mu^*(\G_{C + x_1+\cdots+x_i}) +1.$$
Since $\mu^*(\G_C) = \mu^*(C) = 1$, this implies $\mu^*(\G_{C \cup P}) \le s+1$. 
Therefore, the conclusion follows from Theorem \ref{LT}.
\end{proof}

Let $I$ be a monomial ideal. Then $R/I$ has a natural $\NN^n$-graded structure.
Takayama \cite{Ta} showed that the $\ZZ^n$-graded components of the local cohomology modules of $R/I$ can be described in terms of certain complexes, which are defined as follows.
 
For every vector $\a = (a_1,...,a_n) \in \NN^n$, let $x^\a = x_1^{a_1}\cdots x_1^{a_n}$ and
$$\D_\a(I)  := \{F \subseteq V|\ x^\a \not\in IR[x_i^{-1}|\ i \in F]\}.$$
Then $\D_\a(I)$ is a simplicial complex, which we call a {\em degree complex} of $I$ because the
$\a$-graded component of the local cohomology modules of $R/I$ depends on the reduced cohomology of $\D_\a(I)$. The above definition of $\D_\a(I)$ is due to \cite[Lemma 1.2]{MT}, which is simpler than the original definition in \cite{Ta}. \par

We have the following technical criterion for $H_\mm^1(R/I) \neq 0$. 
For every vertex $v \in V$ let $R_v := k[x_i|\ i \neq v]$ and $I_v := IR[x_v^{-1}]\cap R_v$.

\begin{Proposition}  \label{TT} \cite[Proposition 1.6]{TT}
$H_\mm^1(R/I) \neq 0$ if and only if one of the following conditions are satisfied:\par
{\rm (1)} $\depth R_v/I_v  =  0$ for some $v \in V$.\par
{\rm (2)} $\D_\a(I)$ is disconnected for some $\a \in \NN^n$. 
\end{Proposition}

Now we will explain how to check the conditions (1) and (2)  for the ideal $I^t$, where $I$ is the edge ideal of a graph $\G$. 
Let $N(v)$ denote the set of all vertices adjacent to a vertex $v$ and $N[v] = N(v) + v$, which is called  the closed neighborhood of $v$.

\begin{Lemma} \label{reduction}
Let $J$ be the edge ideal of $\G_{V \setminus N[v]}$ in $S := k[x_i|\ i \not\in N[v]]$.
Then $\depth R_v/I_v^t  =  0$ if and only if $\depth S/J^t  =  0$.
\end{Lemma}

\begin{proof} 
Let $\pp$ be the ideal of $R$ generated by the variables $x_i$, $i \neq v$. Then $\pp R_v$ is the maximal homogeneous ideal of $R_v$. We have $\depth R_v/I_v^t  =  0$ if and only if $\pp R_v$ is an associated prime of $I_v^t$. Since $I_v$ is a localization of $I$, this condition is satisfied if and only if  $\pp$ is an associated prime of $I^t$. By \cite[Lemma 3.4]{HM}, $\pp$ is an associated prime of $I^t$ if and only if $\depth S/J^s  =  0$ for some $s \le t$. By Corollary \ref{zero}, $\depth S/J^s  =  0$ implies $\depth S/J^t  =  0$.
\end{proof}

If $\depth S/J^t  =  0$, $\G_{V \setminus N[v]}$ must be a strongly non-bipartite graph. In this case, we can use Theorem \ref{LT} to check when $\depth S/J^t  =  0$.
\par

For a subset $F \subseteq V$, we denote by $P_F$ the ideal $(x_i|\ i \not\in F)$ of $R$.

\begin{Lemma} \label{facet1}
$P_F$ is an associated prime of $I^t$ if $F$ is a facet of $\D_\a(I^t)$.
\end{Lemma}

\begin{proof}
Let $I_F := IR[x_i^{-1}|\ i \in F] \cap k[x_i|\ i \not\in F]$. If $F$ is a facet of $\D_\a(I^t)$,
$x^\a \in \widetilde{I_F^t} \setminus I_F^t$, 
where $\widetilde{I_F^t}$ is the saturation of $I_F^t$ \cite[Lemma 1.3]{TT}.
From this it follows that $P_F$ is an associated prime of $I^t$ \cite[Lemma 1.2]{HLT}. 
\end{proof}

It is obvious that $I \subseteq P_F$ if and only if $F$ is an independent set of $\G$.
Let $\max(\G)$ denote the set of the maximal independent sets of $\G$. Then
$$I^{(t)} = \bigcap_{F \in \max(\G)} P_F^t.$$

\begin{Lemma} \label{facet2}
A set $F \subset V$ is a facet of $\D_\a(I^{(t)})$ if and only if 
$F$ is a maximal independent set of $\G$ and $\sum_{i \not\in F}a_i < t$.
\end{Lemma}

\begin{proof} 
By \cite[Lemma 1.5]{MT}, $F$ is a maximal facet of $\D_\a(I^{(t)})$ if and only if $F$ is a maximal independent set of $\G$ and $x^\a \not\in P_F^t$. It is obvious that $x^\a \not\in P_F^t$ if and only if $\sum_{i \not\in F}a_i < t$.
\end{proof}

We will use the following lemma to show the existence of disconnected degree complexes.\par

\begin{Lemma}\label{TNT}
Let $\G$ be a connected bipartite graph.  Let  $(F,G)$ be the partition of $\G$. 
For $t \ge n$, there exists $\a \in \NN^n$ with $a_i > 0$ for all $i \in V$ such that
$F,G$ are the only facets of $\D_\a(I^t)$ and
$$\sum_{i \in F}a_i = \sum_{i \in G}a_i = t-1.$$
\end{Lemma}

\begin{proof} 
Recall that a vertex $i \in V$ is a leaf of $\G$ if $\deg(i) = 1$, where $\deg(i)$ is the number of adjacent vertices to $i$. A leaf edge is an edge incident to a leaf. Let $\varepsilon_0(\G)$ denote the numbers of leaf edges of $\G$.
By the proof of \cite[Lemma 3.4]{TNT} we know that for some $s \le n - \varepsilon_0(\G)$, there exists $\b = (b_1,...,b_n) \in \NN^n$ such that
$F,G$ are the only facets of $\D_\b(I^s)$ and
$$\sum_{i \in F}b_i = \sum_{i \in G}b_i = s-1.$$
In particular, $b_j = 0$ if $j$ is incident only to leaf edges.
Since $\G$ is connected, two different vertices $i,j$ are not incident to the same leaf edge.
Therefore, the number of indices $j$ with $b_j = 0$ is equal or less than $\varepsilon_0(\G)$.
\par

Let $\e_j$ denote the $j$-th unit vector of $\NN^n$.
If $b_j = 0$, choose a vertex $v$ adjacent to $i$.
Set $\a = \b + \e_j +\e_v$. Then $a_j = 1$.
Let $H$ be an arbitrary maximal independent set of $\G$. Since $H$ does not contain both $j,v$, we have
$$\sum_{i \not\in H}a_i \le  \sum_{i \not\in H}b_i + 1.$$
This inequality becomes an equality if $H = F$ or $H = G$ because $F,G$ contain $j$ or $v$.
Note that $I^t = I^{(t)}$ for all $t \ge 1$ \cite[Theorem 5.9]{SVV}.  
Using Lemma \ref{facet2} we can check that $H$ is a facet of $\D_\a(I^{s+1})$ if and only if $H$ is a facet of $\D_\b(I^s)$.
From this it follows that $\D_\a(I^{s+1}) = \D_\b(I^s)$ and
$$\sum_{i \in F}a_i = \sum_{i \in G}a_i = s.$$

For $t \ge n$ we have $t-s \ge \varepsilon_0(\G)$. 
By repeating the above technique $t-s$ times we will reach the conclusion.
\end{proof}


\section{Existence of a power $I^t$ with $H_\mm^1(R/I^t) \neq 0$}

\begin{Theorem} \label{non-dominating1}
Let $\G$ be a connected non-bipartite graph which has at least a non-dominating odd cycle.
Assume that there exist an odd cycle $C$ of length $2r+1$ and a vertex $v$ not adjacent to $C$ such that 
the induced subgraph $\G_{V \setminus  N[v]}$ is connected. Then  \par
{\rm (1)} $H_\mm^1(R/I^t) \neq 0$ for $t \ge r+1$ if $C = V \setminus N[v]$, \par
{\rm (2)} $H_\mm^1(R/I^t) \neq 0$ for $t \ge n-|N[v]|-r-1$ if $C \neq V \setminus N[v]$.
\end{Theorem}

\begin{proof} 
By Proposition \ref{TT},  $H_\mm^1(R/I^t) \neq 0$ if $\depth R_v/I_v^t  =  0$.
Let $J$ be the edge ideal of $\G_{V \setminus  N[v]}$ in $S := k[x_i|\ i \not\in N[v]]$.
By Lemma \ref{reduction}, $\depth R_v/I_v^t  =  0$ if $\depth S/J^t = 0$. 
Since $v$ is not adjacent to $C$, $\G_{V \setminus  N[v]}$ contains $C$. 
Thus, $\G_{V \setminus  N[v]}$ is a non-bipartite graph. 
By Theorem \ref{LT}, $\depth S/J^t = 0$ for $t \ge s(\G_{V \setminus  N[v]})+1$.
\par
If $C = V \setminus N[V]$, we have $\mu^*(\G_C) \le \mu^*(C) = r$ \cite[Lemma 5.1]{LT}. 
Hence, $s(\G_C) \le r.$ 
\par
If $C \neq V \setminus N[V]$, 
there exists a subset $U \subseteq V \setminus N[V]$ which differs from $V \setminus N[V]$ by only a vertex
such that $\G_U$ contains $C$ and is connected. Obviously, $U$ is a dominating set of $\G_{V \setminus  N[v]}$.
Since $\G_U$ can be seen as the extension of $C$ by $|U| - 2r-1$ edges, we can use Proposition \ref{extension} to deduce that 
$$\mu^*(\G_U) \le \mu^*(\G_C) + |U| - 2r-1 \le \mu^*(C) + |V \setminus N[v]|-2r-2 = n-N[v]-r-2.$$
From this it follows that $s(\G_{V \setminus  N[v]}) \le n-|N[v]|-r-2.$
\end{proof}

\begin{Remark} 
The range $t \ge r+1$ or $t \ge n-|N[v]|-r-1$ for $H_\mm^1(R/I^t) \neq 0$ in Theorem \ref{non-dominating1} is not the best possible. For case (1) let $\G$ be the graph in Figure 2. Then $r +1 = 2$, whereas $H_\mm^1(R/I) \neq 0$ by Proposition \ref{depth R/I} (there is a vertex adjacent to all other vertices). For case (2) let $\G$ be the graph in Figure 1. Then $n-|N[v]|-r-1 = 3$, whereas $H_\mm^1(R/I^2) \neq 0$ by Example \ref{best}.
\end{Remark} 

\begin{figure}[ht!]

\begin{tikzpicture}[scale=0.5] 

\draw [thick] (0,0) coordinate (a) -- (0,2) coordinate (b) ;
\draw [thick] (0,2) coordinate (b) -- (1.5,1) coordinate (c) ;
\draw [thick] (1.5,1) coordinate (c) -- (0,0) coordinate (a) ; 
\draw [thick] (1.5,1) coordinate (c) -- (3.5,1) coordinate (d);
\draw [thick] (3.5,1) coordinate (d) -- (5.5,1) coordinate (e);
\draw [thick] (3.5,1) coordinate (d) -- (0,0) coordinate (a);
\draw [thick] (3.5,1) coordinate (d) -- (0,2) coordinate (b);

\fill (a) circle (3pt);
\fill (b) circle (3pt);
\fill (c) circle (3pt);
\fill (d) circle (3pt);
\fill (e) circle (3pt);
 
\end{tikzpicture}
\caption{}
\end{figure}

\begin{Remark} 
The additional assumption of Theorem \ref{non-dominating1} seems to be always satisfied. However, that is not the case;  see the graph in Figure 3. It would be of interest to characterize connected non-bipartite graphs such that $\G_{V \setminus  N[v]}$ is disconnected for any vertex $v$ not adjacent to any odd cycle.
\end{Remark} 

\begin{figure}[ht!]

\begin{tikzpicture}[scale=0.5] 

\draw [thick] (0,0) coordinate (a) -- (0,2) coordinate (b) ;
\draw [thick] (0,2) coordinate (b) -- (1.5,1) coordinate (c) ;
\draw [thick] (1.5,1) coordinate (c) -- (0,0) coordinate (a) ; 
\draw [thick] (1.5,1) coordinate (c) -- (3.5,1) coordinate (d);
\draw [thick] (3.5,1) coordinate (d) -- (5,2) coordinate (e);
\draw [thick] (3.5,1) coordinate (d) -- (5,0) coordinate (f);

\fill (a) circle (3pt);
\fill (b) circle (3pt);
\fill (c) circle (3pt);
\fill (d) circle (3pt);
\fill (e) circle (3pt);
\fill (f) circle (3pt);
 
\end{tikzpicture}
\caption{}
\end{figure}

If $\G$ has only dominating odd cycles, we are able to give a criterion for the existence of $H_\mm^1(R/I^t) \neq 0$. This criterion is a consequence of the following criterion for the existence of a power $I^t$ with $\depth R/I^{(t)} = 1$.

\begin{Theorem} \label{symbolic1}
Let $\G$ be a connected graph. Then $\depth R/I^{(t)} = 1$ for some $t \ge 1$ if and only if $\G$ has two disjoint maximal independent sets $F,G$ such that the induced graph $\G_{F \cup G}$ is connected.
In this case, $\depth R/I^{(t)} = 1$ for $t \ge |F| + |G|$.
\end{Theorem}

\begin{proof}
Since $\mm$ is not an associated prime of $I^{(t)}$, $H_\mm^0(R/I^{(t)}) = 0$. 
Therefore, $\depth R/I^{(t)} = 1$ if and only if $H_\mm^1(R/I^{(t)}) \neq 0$.
\par

If there exists a vertex $v$ such that $N[v] = V$, 
then $x_i \in IR[x_v^{-1}]$ for all $i \neq v$ because $x_ix_v \in I$.
Hence, $I_v = IR[x_v^{-1}] \cap R_v$ is the maximal homogeneous ideal of $R_v = k[x_i|\ i \neq v]$. 
Therefore, $\dim R_v/I_v^t = 0$, which implies $H_\mm^1(R/I^{(t)}) \neq 0$ for all $t \ge 1$ by Proposition \ref{TT}.
In this case, $\{v\}$ and any maximal independent set $F$ of $\G_v$ are two disjoint maximal independent sets of $\G$ and $\G_{F+v}$ is connected.
\par 

If $N[v] \neq V$ for all vertices $v$, then $I_v$ is not the maximal homogeneous ideal of $R_v$. Hence $\dim R/I_v^t > 0$. Therefore, $\depth R/I_v^{(t)} > 0$ because $I^{(t)}_v$ is an unmixed ideal for all $t \ge 1$.
By Proposition \ref{TT}, $H_\mm^1(R/I^{(t)}) \neq 0$ if and only if there exists $\a \in \NN^n$ such that $\D_\a(I^{(t)})$ is disconnected. It suffices to show that there exists such $\a \in \NN^n$ for some $t \ge 1$ if and only if $\G$ has two disjoint maximal independent sets $F,G$ such that the induced graph $\G_{F \cup G}$ is connected.

Suppose that there exists $\a \in \NN^n$ such that $\D_\a(I^{(t)})$ is disconnected for some $t \ge 1$. 
Then $\D_\a(I^{(t)})$ has two disconnected facets $F,G$. 
By Lemma \ref{facet2}, $F,G$ are maximal independent sets of $\G$ and
$$ \sum_{i \not\in F}a_i \le t-1,\; 
\sum_{i \not\in G}a_i \le t-1.$$
From this it follows that 
$$\sum_{i \not\in F}a_i + \sum_{i \not\in G}a_i  \le 2(t-1).$$

Suppose that $\G_{F \cup G}$ is disconnected. Let $F \cup G = V_1 \sqcup V_2$ be a partition such that 
$\G_{F \cup G}$ is the disjoint union of the subgraphs $\G_{V_1}$ and $\G_{V_2}$.
For $i = 1,2$ let $F_i = F \cap V_i$ and $G_i = G \cap V_i$. Then $F = F_1\sqcup F_2$ and $G = G_1 \sqcup G_2$.
Hence $F_1 \cup G_2$ and $F_2 \cup G_1$ are maximal independent sets with
\begin{align*}
\sum_{i \not\in F_1 \cup G_2}a_i + \sum_{i \not\in F_2 \cup G_1}a_i  & = \sum_{i \not\in F_1}a_i + \sum_{i \not\in G_2}a_i + \sum_{i \not\in F_2}a_i + \sum_{i \not\in G_1}a_i\\
& = \sum_{i \not\in F}a_i + \sum_{i \not\in G}a_i.
\end{align*}
Therefore, $\displaystyle \sum_{i \not\in F_1 \cup G_2}a_i \le t-1$ or $\displaystyle \sum_{i \not\in F_2 \cup G_1}a_i \le t-1$.
By Lemma \ref{facet2}, this implies that $F_1 \cup G_2$ or $F_2 \cup G_1$ is a facet of $\D_\a(I^{(t)})$. Since $F$ and $G$ are connected via $F_1 \cup G_2$ or $F_2 \cup G_1$, we would get a contradiction to the disconnectedness of $F,G$ in $\D_\a(I^{(t)})$. \par

Conversely, assume that $\G$ has two disjoint maximal independents sets $F,G$ such that $\G_{F \cup G}$ is connected. Then $\G_{F \cup G}$ is a bipartite graph with the bipartition $(F,G)$.
Let $J$ denote the edge ideal of $\G_{F \cup G}$ and $r = |F \cup G|$. 
By Lemma \ref{TNT}, for $t \ge |F|+|G|$, there exists $\b \in \NN^r$ with $b_i > 0$ for all $i$ such that $F, G$ are the only facets of $\D_\b(J^t)$ and
$\displaystyle \sum_{i \in F} b_i = \sum_{i \in G} b_i = t-1.$
Embed $\NN^r$ in $\NN^n$ according to the indices of the vertices of $F \cup G$ in $V$.
Let $\a \in \NN^n$ such that $a_i = b_i$ if $i \in F\cup G$ and $a_i = 0$ else. 
Then $F, G$ are facets of $\D_\a(I^{(t)})$ by Lemma \ref{facet2}. 
Since $J^t \subset I^{(t)}$, any face of $\D_\a(I^{(t)})$ contained in $F \cup G$ is also a face of $\D_\b(J^t)$ by the definition of the degree complex. \par

Suppose that $\D_\a(I^{(t)})$ has a facet $H \neq F,G$. 
Then $H$ does not contain vertices of both $F,G$ because these vertices would form a face of $\D_\b(J^t)$ which connect $F$ with $G$. 
Without loss of generality we may assume that $H \cap F = \emptyset$. Then 
$H \neq G$. Thus,  
$$\sum_{i \not\in H} a_i \ge \sum_{i \in F} a_i + \sum_{i \in G - H}a_i = \sum_{i \in F} b_i + \sum_{i \in G - H}b_i> \sum_{i \in F} b_i = t-1.$$
By Lemma \ref{facet2}, this implies $H \not\in \D_\a(I^{(t)})$, a contradiction. 
So $F,G$ are the only facets of $\D_\a(I^t)$. Since $F,G$ are disjoint, $\D_\a(I^{(t)})$  is disconnected. By Proposition \ref{TT}, this implies $H_\mm^1(R/I^{(t)}) \neq 0$ for $t \ge |F|+|G|$.
\end{proof}

\begin{Remark} \label{RTY}
By \cite[Theorem 3.2]{RTY}, $\depth R/I^{(2)} = 1$ if and only if $\diam(\G^c) \ge 3$, where $\G^c$ denotes the graph of the non-edges of $\G$, and $\diam(\G^c)$ denotes the diameter of $\G^c$, which is the maximum of the distance between two vertices of $\G^c$. Therefore, if 
$\diam(G^c) \ge 3$, there always exists a pair of disjoint maximal independent sets $F,G$ such that the induced graph $\G_{F \cup G}$ is connected. 
\end{Remark}

\begin{Theorem} \label{dominating}
Let $\G$ be a  connected non-bipartite graph  which has only dominating odd cycles. Then $H_\mm^1(R/I^t) \neq 0$ for some $t \ge 1$ if and only if $\G$ has two disjoint maximal independent sets $F,G$ such that the induced graph $\G_{F \cup G}$ is connected. In this case, $H_\mm^1(R/I^t) \neq 0$ for $t \ge  |F|+|G|$.
\end{Theorem}

\begin{proof} 
First we will show that every associated prime $\pp \neq \mm$ of $I^t$ is a minimal prime of $I$ for all $t \ge 1$. 
Let $v$ be a vertex such that $x_v \not\in \pp$.
Then $\pp R[x_v^{-1}]$ is an associated prime of $I^tR[x_v^{-1}]$.
For $i \in N(v)$, we have $x_ix_v \in I$. Hence, $x_i \in IR[x_v^{-1}]$.
Let $J$ be the edge ideal of $\G_{V \setminus N[v]}$.
Then $IR[x_v^{-1}]$ is generated by $J$ and the variables $x_i$, $i \in N(v)$. 
Therefore, $\pp$ is an associated prime of $(J+(x_i|\ i \in N(v))^t$. \par

Since $v$ is adjacent to all odd cycles of $\G$, the graph $\G_{V \setminus N[v]}$ has no odd cycles. 
Hence, $\G_{V \setminus N[v]}$ is bipartite. From this it follows that $J^t = J^{(t)}$, which is an unmixed ideal for all $t \ge 1$ \cite[Theorem 5.9]{SVV}.  Since $(x_i|\ i \in N(v))^t$ is an unmixed ideal, $(J+(x_i|\ i \in N(v))^t$ is also an unmixed ideal \cite[Lemma 3.4]{HM}. Therefore, $\pp$ is a minimal prime of $(J+(x_i|\ i \in N(v))^t$.
Since $IR[x_v^{-1}] = (J+(x_i|\ i \in N(v))^tR[x_v^{-1}]$ and $x_v \not\in \pp$, this implies that $\pp$ is a minimal prime of $I$. \par

Since every associated prime $\pp \neq \mm$ of $I^t$ is a minimal prime of $I$, $I^{(t)}$ is the saturation of $I^t$.
Thus, $I^{(t)}/I^t$ has finite length. 
From the exact sequence
$$0 \rightarrow I^{(t)}/I^t \rightarrow R/I^t \rightarrow R/I^{(t)} \rightarrow 0$$
we deduce that $H_\mm^1(R/I^t) \cong H_\mm^1(R/I^{(t)})$.
Since $H_\mm^1(R/I^{(t)}) \neq 0$ means $\depth R/I^{(t)} = 1$, 
 the conclusion follows from Theorem \ref{symbolic1}.
\end{proof}

\begin{Remark}
The range $t \ge |F|+|G|$ for $H_\mm^1(R/I) \neq 0$ in Theorem \ref{dominating} is not the best possible. Let $\G$ be a triangle. Then $F,G$ consist of only one vertex. Hence $|F|+|G| = 2$, whereas $H_\mm^1(R/I) \neq 0$ by Proposition \ref{depth R/I}.  
\end{Remark}

The condition  $H_\mm^1(R/I) \neq 0$ for $t \gg 1$ does not imply the existence of a power $I^t$ with $\depth R/I^t = 1$.

\begin{Example}  
Let $\G$ be the pentagon. Then $\dim R/I = \depth R/I = \depth R/I^2 =  2$ by \cite[Corollary 3.4]{MT1} ($I$ is the Stanley-Reisner ideal of itself). By Theorem \ref{LT}, $\depth R/I^t = 0$ for $t \ge 3$, whereas $H_\mm^1(R/I^t) \neq 0$ for $t \gg 1$ by Theorem \ref{dominating}.
\end{Example}


\section{Persistence of the condition $H_\mm^1(R/I^t) \neq 0$}

By Proposition \ref{TT}, $H_\mm^1(R/I^t) \neq 0$ if and only if one of the following conditions is satisfied 
\begin{enumerate}
\item $\depth R_v/I^t_v = 0$ for some $v$.
\item There exists $\a \in \NN^n$ such that $\D_\a(I^t)$ is disconnected.
\end{enumerate}
In order to investigate the persistence of the condition $H_\mm^1(R/I^t) \neq 0$ we need to investigate whether the corresponding condition also holds for the ideal $I^{t+1}$. 
 
 \begin{Proposition} \label{persistence1} 
Assume that $\depth R_v/I^t_v = 0$. Then $\depth R_v/I^{t+1}_v = 0$. 
\end{Proposition}

\begin{proof} 
Let $J$ be the edge ideal of the graph $\G_{V \setminus N[v]}$ in the polynomial ring $S := k[x_i|\ i\not\in N[v]]$. 
By Lemma \ref{reduction}, $\depth R_v/I^t_v = 0$ implies $\depth S/J^t = 0$.
By Corollary \ref{zero}, $\depth S/J^t = 0$ implies $\depth S/J^{t+1}= 0$.
Hence $\depth R_v/I^{t+1}_v = 0$ by Lemma \ref{reduction}. 
\end{proof} 

In the following we say that two sets of vertices is adjacent in $\G$ if there is an edge connecting them.  

\begin{Proposition} \label{persistence2} 
Assume that there exists $\a \in \NN^n$ such that $\D_\a(I^t)$ is disconnected. If $\D_\a(I^t)$ has disconnected components which are adjacent in $\G$, then there exists $\b \in \NN^n$ such that $\D_\b(I^{t+1})$ is disconnected.
\end{Proposition}

\begin{proof} 
Let $\{u,v\}$ be an edge of $\G$ which connect two disconnected components of $\D_\a(I^{t})$ which are adjacent in $\G$. Let $\b = \a+\e_u+\e_v$.

Let $F$ be an arbitrary face of $\D_\b(I^{t+1})$. 
Then $x^\b \not\in I^{t+1}R[x_i^{-1}|\ i \in F]$ by the definition of the degree complex.
If $F$ is not a face of $\D_\a(I^t)$, then $x^\a \in I^tR[x_i^{-1}|\ i \in F]$. Since $x_ux_v \in I$, 
$x^\b = x^\a(x_ux_v) \in I^{t+1}R[x_i^{-1}|\ i \in F],$
which is a contradiction. Therefore, $F$ is a face of $\D_\a(I^t)$. 
Consequently, $\D_\b(I^{t+1})$ is a subcomplex of $\D_\a(I^t)$.
Since $u,v$ belong to different connected components of $\D_\a(I^t)$, 
$\D_\b(I^{t+1})$ is disconnected if $u,v \in \D_\b(I^{t+1})$.\par

Suppose that $u \not\in \D_\b(I^{t+1})$. 
Then $x^\b \in I^{t+1}R[x_u^{-1}]$. 
Hence 
$$x^\a \in I^{t+1}R[x_u^{-1}]: x_ux_v =  I^{t+1}R[x_u^{-1}]: x_v.$$
Since $x_ux_v \in I$, we have $x_v \in IR[x_u^{-1}]$. 
Since $I$ is a monomial ideal and $x_v$ is a variable,
$I^{t+1}R[x_u^{-1}]: x_v = I^tR[x_u^{-1}].$ 
Therefore, $x^\a \in I^tR[x_u^{-1}]$. This implies $u \not\in \D_\a(I^t)$, a contradiction.
So we can conclude that $u \in \D_\b(I^{t+1})$.
Similarly, we also have $v \in \D_\b(I^{t+1})$.
\end{proof}

\begin{Remark} \label{rem}
The additional condition on $\D_\a(I^t)$ in Theorem \ref{persistence2} is satisfied if $\D_\a(I^t)$ has a facet which is a maximal independent set. This follows from the fact that a maximal independent set $F$ is adjacent to every vertex outside of $F$. 
\end{Remark}

\begin{Theorem}  \label{dominating2}
Let $\G$ be a connected non-bipartite graph which has only dominating odd cycles. 
If $H_\mm^1(R/I^t) \neq 0$, then $H_\mm^1(R/I^{t+1}) \neq 0$. 
\end{Theorem}

\begin{proof} 
If $H_\mm^1(R/I^t) \neq 0$ for some $t > 0$, then one of the conditions (1) and (2) of Proposition \ref{TT} is satisfied.
We only need to show that the corresponding condition holds for $I^{t+1}$.
If (1) is satisfied, the conclusion follows from Proposition \ref{persistence1}. \par

If (2) is satisfied, the conclusion follows from Proposition \ref{persistence2} if $\D_\a(I^t)$ has disconnected components which are adjacent in $\G$. By Remark \ref{rem}, this condition is satisfied if $\D_\a(I^t)$ has a facet which is a maximal independent set.

Let $F$ be an arbitrary facet of $\D_\a(I^t)$. 
Then $P_F := (x_i|\ i \not\in F)$ is an associated prime of $I^t$ by Lemma \ref{facet1}.
Since $F \neq \emptyset$, $P_F \neq \mm$. 
Hence $P_F$ is a minimal prime of $I$ by the proof of Theorem \ref{dominating}. 
Therefore, $F$ is a maximal independent sets of $\G$.  
\end{proof}

Our approach also yields the following result on symbolic powers, which is known before only for connected bipartite graph \cite[Lemma 1.3 (2)]{TNT}.

\begin{Theorem}  \label{symbolic2}
Let $\G$ be a connected graph. 
If $\depth (R/I^{(t)}) = 1$ for some $t \ge 1$, then $\depth (R/I^{(t+1)}) = 1$. 
\end{Theorem}

\begin{proof} 
Note that $\depth (R/I^{(t)}) = 1$ if and only if $H_\mm^1(R/I^{(t)}) \neq 0$.
By the proof of Theorem \ref{symbolic1}, we may assume that $N[v] \neq V$ for every vertex $v \in V$ and that 
there exists $\a \in \NN^n$ such that $\D_\a(I^{(t)})$ is disconnected. 

Let $F,G$ be facets of $\D_\a(I^{(t)})$ which are disconnected.
Since $F,G$ are maximal independent sets by Lemma \ref{facet2}, there exist adjacent vertices $u \in F$ and $v \in G$. Set $\b = \a+\e_u+\e_v$. We will show that $\D_\b(I^{(t+1)})$ is disconnected, which implies 
$H_\mm^1(R/I^{(t+1)}) \neq 0$ by Proposition \ref{TT}. \par

Let $H$ be an arbitrary face of $\D_\b(I^{(t+1)})$. 
Then $x^\b \not\in I^{(t+1)}R[x_i^{-1}|\ i \in H]$.
If $H$ is not a face of $\D_\a(I^{(t)})$, then $x^\a \in I^{(t)}R[x_i^{-1}|\ i \in H]$. Since $x_ux_v \in I$, 
$$x^\b= x^\a(x_ux_v) \in I^{(t)}IR[x_i^{-1}|\ i \in H] \subseteq I^{(t+1)}R[x_i^{-1}|\ i \in H],$$ 
which is a contradiction. Therefore, $H$ is a face of $\D_\a(I^{(t)})$. 
Consequently, $\D_\b(I^{(t+1)})$ is a subcomplex of $\D_\a(I^{(t)})$.
From this it follows that $\D_\b(I^{(t+1)})$ is disconnected if it contains $F,G$ (which belong to different connected components of $\D_\a(I^{(t)})$). \par

By Lemma \ref{facet2}, $F$ is a maximal independent set of $\G$ with $\sum_{i \not\in F}a_i < t$.
Since $u \in F$ and $v \not\in F$, 
$\sum_{i \not\in F}b_i = \sum_{i \not\in F}a_i + 1 <  t+1$. Hence, $F \in \D_\b(I^{(t+1)})$.
Similarly, we also have $G \in \D_\b(I^{(t+1)})$.
\end{proof}

With regard to the above results we raise the following problem, which we are unable to find a counter-example.

\begin{Conjecture} 
Let $\G$ be a connected graph. 
If $H_\mm^1(R/I^t) \neq 0$ for some $t > 0$, then $H_\mm^1(R/I^{t+1}) \neq 0$. 
\end{Conjecture} 


\section{Decreasing behavior after $\depth R/I^t = 1$}

Let $I$ be the edge ideal of a graph $\G$.
If $\G$ is a connected non-bipartite graph, we know that $\depth R/I^t = 0$ for $t \gg 1$. 
If furthermore $\depth R/I^t = 1$ for some $t \ge 1$, we want to find a number 
 $\delta$ independent of $\G$ such that $\depth R/I^{t+\delta} = 0$.  

First, we will study the case $\depth R/I = 1$. 

\begin{Proposition} \label{depth R/I} 
$\depth R/I = 1$ if and only if there is a partition $V = A \sqcup B$ such that every vertex of $A$ is adjacent to every vertex $B$.
\end{Proposition}

\begin{proof}
Since $I$ is an unmixed ideal, $\depth R/I = 1$ if and only if $H_\mm^1(R/I) \neq 0$.
Let $\D$ be the complex of the independent sets of $\G$.
The edges of $\G$ are minimal among the sets not contained in $\D$.
Hence $I$ is the Stanley-Reisner ideal of $\D$. 
By Hochster’s formula for local cohomology of Stanley-Reisner ideals \cite{Hoc}, 
$H_\mm^1(R/I) \neq 0$ if and only if $\D$ is disconnected.
This condition is satisfied if and only if there is a partition $V = A \sqcup B$ 
such that every pair of vertices $u \in A, v \in B$ is dependent or, equivalently, $\{u,v\} \in \G$. 
\end{proof}

\begin{Theorem} \label{decrease1}
Let $\G$ be a non-bipartite graph.
If $\depth R/I = 1$, then $\depth R/I^2= 0$.
\end{Theorem}

\begin{proof}
By Theorem \ref{depth R/I}, we may assume that there is a partition $V = A \sqcup B$ such that every vertex of $A$ is adjacent to every vertex of $B$. As a consequence, every set containing vertices in both $A$ and $B$ is dominating in $\G$. 
Since $\G$ is non-bipartite, $A$ and $B$ can not be both independent sets. 
So we may assume that $A$ contains two adjacent vertices $u,v$. 
Let $w$ be a vertex of $B$. Then $u,v,w$ form a triangle of $\G$. 
Since $\{u,v,w\}$ is a dominating set of $\G$, $\depth R/I^2 = 0$ by Theorem \ref{LT}.
\end{proof}

Next, we consider the case $\depth R/I^2 = 1$. For that we already have the following criterion, which is a direct consequence of \cite[Theorem 4.8]{TT}.

\begin{Theorem} \label{square2} 
$\depth R/I^2 \le 1$ if and only if one of the following conditions is satisfied:\par
{\rm (1)} $\diam(\G^c) \ge 3$.\par
{\rm (2)}  There exists a triangle $C$ of $\G$ such that $V \setminus N[C]$ has at most one element or the induced graph $\G^c_{V \setminus N[C]}$ is not connected.
\end{Theorem}

\begin{Theorem}\label{square1}
Let $\G$ be a connected non-bipartite graph.
If $\depth R/I^2 = 1$, then $\depth R/I^5 = 0$.
\end{Theorem}

The proof of Theorem \ref{square1} follows from the following propositions, which deal with the cases (1) and (2) of Theorem \ref{square2} separately.

\begin{Proposition} 
Let $\G$ be a connected non-bipartite graph. Assume that $\diam(\G^c) \ge 3$.
Then $\depth R/I^3 = 0$.
\end{Proposition}

\begin{proof}
For every vertex $i$ let $N_i := V \setminus N(i)$, which is the set of the non-adjacent vertices of $i$.
Then $N_i$ is the closed neighborhood of $i$ in $\G^c$. 
The assumption $\diam(\G^c) \ge 3$ implies that $\G$ has an edge $\{i,j\}$ such that $N_i \cap N_j = \emptyset$.
From this it follows that every vertex of $V$ is adjacent to $i$ or $j$. Hence $\{i,j\}$ is a dominating set of $\G$.
In particular, $i$ is adjacent to all vertices of $N_j$ and $j$ is adjacent to all vertices of $N_i$. 
\par
If $N_i \cup N_j \neq V$, every vertex $v \not\in N_i \cup N_j$ is adjacent to both $i,j$. Hence, $i,j,v$ form a triangle of $\G$. Since $\{i,j,v\}$ is a dominating set of $\G$, $\depth R/I^3 = 0$ by Theorem \ref{LT}. \par
If $N_i \cup N_j = V$, then $N_i, N_j$ can not be both independent sets because otherwise $\G$ would be bipartite.
So we may assume that $N_i$ contains two adjacent vertices $u,v$. Then $\{j,u,v\}$ form a triangle $C$ of $\G$. 
If $i$ is one of the vertices $u,v$, say $i = u$, then $C$ is a dominating set of $\G$. Thus, $\depth R/I^3 = 0$ by Theorem \ref{LT}. If $i$ is not one of the vertices $u,v$, $C \cup \{i,j\}$ is a dominating set of $\G$. Therefore, $\depth R/I^3 = 0$ by Lemma \ref{triangle}.
\end{proof}

\begin{Proposition}
Let $\G$ be a connected non-bipartite graph. Assume that there is a triangle $C$ of $\G$ such that either $V \setminus N[C]$ has at most one vertex or $\G_{V \setminus N[C]}^c$ is disconnected. Then $\depth R/I^5 = 0$.
\end{Proposition}

\begin{proof}
If $V = N[C]$, then $C$ is a dominating set of $\G$. This implies $\depth R/I^5 = 0$ by Theorem \ref{LT}. 
If $V \setminus  N[C]$ has only a vertex $u$, then $C+u$ is a dominating set of $\G$.
Since $\G$ is connected, $u$ is adjacent to a vertex $v \in N[C]$. 
Let $w \in C$ be a vertex adjacent to $v$. Then $\{u,v,w\}$ form a path of $\G$. Since $C\cup \{u,v,w\}$ is a dominating set in $\G$, $\depth R/I^5 = 0$ by Lemma \ref{triangle}. 
\par
It remains to prove the case $\G_{V \setminus N[C]}^c$ is disconnected. 
Let $V \setminus N[C] = V_1\sqcup V_2$ be a partition such that the induced graphs of $\G_{V \setminus N[C]}^c$ on $V_1,V_2$ are disconnected. Then every vertex of $V_1$ is adjacent to every vertex of $V_2$ in $\G$. Since $\G$ is connected, there exists a vertex in $V_1 \cup V_2$, say $u_1 \in V_1$, which is adjacent to a vertex $v \in N[C]$. Let $w$ be a vertex of  $C$ adjacent to $v$. Let $u_2$ be a vertex of  $V_2$ adjacent to $u_1$. Then $w,v,u_1,u_2$ form a path $P$ meeting $C$ at $w$ only.
Since every vertex of $N[C]$ is adjacent to a vertex of $C$ and every vertex of $V \setminus N[C]$ is adjacent either to $u_1$ or $u_2$, $C\cup P$ is a dominating set of $\G$. Therefore, $\depth R/I^5 = 0$ by Lemma \ref{triangle}.
\end{proof}

The conclusion of Theorem \ref{square2} is the best possible. 

\begin{Example} \label{best}
Let $\G$ be the graph in Figure 1.
Then $\depth R/I^t > 0$ for $t \le 4$ and $\depth R/I^5 = 0$ by Example \ref{union}. Since $\G_{V \setminus N[C]}^c$ is disconnected, $\depth R/I^2 = 1$ by Theorem \ref{square2}. By Theorem \ref{non-dominating1}, $\depth R/I^3 = \depth R/I^4 = 1$.
\end{Example}

Finally, we consider one of the main case for $H_\mm^1(R/I^t)\neq 0$; see Proposition \ref{TT}.

\begin{Theorem} \label{decrease}
Let $\G$ be a connected non-bipartite graph. Assume that there exists a vertex $v$ such that $\depth R_v/I_v^t = 0$.
Then $\depth R/I^{t+3} = 0$.
\end{Theorem}

\begin{proof}
Let $J$ denote the edge ideal of the graph $\G_{V \setminus N[v]}$ in $S := k[x_i|\ i\not\in N[v]]$.
By Lemma \ref{reduction}, the condition $\depth R_v/I_v^t = 0$ implies $\depth S/J^t = 0$.
By Theorem \ref{LT}, there is a dominating set $C$ of $\G_{V \setminus N[v]}$ such that $\G_C$ is a strongly non-bipartite graph with $\mu^*(\G_C) < t$. Since every vertex of $(V \setminus N[v]) \setminus C$ is adjacent to a vertex of $C$, $V \setminus N[v] \subseteq N[C]$. From this it follows that every vertex of $V$ is adjacent to $C$ or $v$. Hence $C+v$ is a dominating set of $\G$. \par
Since $\G$ is connected and $V = N[C] \cup N[v]$, there are two adjacent vertices $u \in N[C]$ and $w \in N[v]$.
It may happen that $u \in C$ or $w = v$. In any case, $U := C+u+w+v$ is a dominating set of $\G$ since it contains the dominating set $C+v$. Since $U$ is an extension of $C$ by a path of length at most three, we can apply  Proposition \ref{extension} to conclude that $\G_U$ is a strongly non-bipartite graph with 
$\mu^*(\G_U) \le \mu^*(\G_C)+3 < t+3.$
By Theorem \ref{LT}, this implies $\depth R/I^{t+3} = 0$.
\end{proof}

With regard to the above results we raise the following problem, which we are unable to find a counter-example.

\begin{Conjecture} 
Let $\G$ be a connected non-bipartite graph. If $\depth R/I^t = 1$, then $\depth R/I^{t+3} = 0$. 
\end{Conjecture}

\noindent{\bf Acknowledgements}. 
This work is supported by grant ICRTM01-2020.07 of the International Center for Research and Postgraduate Training in Mathematics. Data sharing not applicable to this article as no datasets were generated or analysed during the current study.


\end{document}